\documentclass[11pt]{amsart}

\usepackage{mathtools,amsmath, amsfonts, amssymb, latexsym, amsthm,
  amscd, cancel, enumerate, rotating, comment, appendix, ulem,
  makecell, paralist,times,graphicx,pgf,geometry,dirtytalk,
appendix,float}
\usepackage{thmtools}
\usepackage[colorlinks=true, pdfstartview=FitV, linkcolor=blue, citecolor=blue, urlcolor=blue]{hyperref}
\usepackage[capitalise, noabbrev]{cleveref}
\usepackage{paralist}
\usepackage{pgf}
\usepackage{times}
\usepackage{enumitem}
\usepackage{caption}
\usepackage{subcaption}
\usepackage[T1]{fontenc}

\newcommand{\EDr}{{\ED}^r}

\usepackage{xcolor}
\usepackage{tikz}
\usepackage{tikz-3dplot}
\usetikzlibrary{shapes}
\usetikzlibrary{snakes}
\usetikzlibrary{arrows}
\usepackage{tikz-cd}
\usetikzlibrary{matrix, decorations.pathmorphing}

\usetikzlibrary{calc}
\usetikzlibrary{decorations.markings}

\tikzset{
  midarrow/.style={
    decoration={markings,mark=at position 0.5 with {\arrow{>}}},
    postaction={decorate}
  }
}

\newif\ifdviwin

\numberwithin{equation}{section}
\definecolor{gold}{rgb}{0.9, 0.7, 0.2}

\theoremstyle{plain}
\newtheorem{theorem}{Theorem}[section]

\newtheorem*{Main Theorem}{Main Theorem}

\newtheorem{proposition}[theorem]{Proposition}
\newtheorem{lemma}[theorem]{Lemma}
\newtheorem{corollary}[theorem]{Corollary}

{\Alph{theorem}}
{\Alph{theorem}}

\newcounter{intro}
\newtheorem{introthm}[intro]{Theorem}

\theoremstyle{definition}

\newtheorem{construction}[theorem]{Construction}
\newtheorem{remark}[theorem]{Remark}

\newtheorem{definition}[theorem]{Definition}
\newtheorem{example}[theorem]{Example}

  \newcounter{numlist} %
  {\end{list}}%

\theoremstyle{remark}

\newtheorem{chunk}[theorem]{}

\numberwithin{equation}{section}

\newcommand{\st}{\colon }

\newcommand{\m}{\mathbf{m}}

\newcommand{\lcm}{\mathrm{lcm}}

\newcommand{\LCM}{\mathrm{LCM}}
\newcommand{\e}{\epsilon}
\newcommand{\ed}[1]{\epsilon_{\mbox{\tiny$\D$},#1}}

\newcommand{\ssm}{\smallsetminus}
\newcommand{\height}{\mathrm{height}}

\newcommand{\ED}{\E_{\D}}

\newcommand{\E}{\mathcal{E}}
\newcommand{\NN}{\mathbb{N}}

\newcommand{\ba}{\mathbf{a}}

\newcommand{\bc}{\mathbf{c}}

\newcommand{\qand}{\quad \mbox{ and } \quad }
\newcommand{\qfor}{\quad \mbox{ for } \quad }
\newcommand{\qforeach}{\quad \mbox{ for each } }
\newcommand{\qor}{\quad \mbox{ or } \quad }

\newcommand{\qwith}{\quad \mbox{with} \quad}

\newcommand{\qforsome}{\quad \mbox{for some} \quad}
\newcommand{\qforall}{\quad \mbox{for all} \quad}

\newcommand{\Erq}{{\E_q}^r}

\newcommand{\Nrq}{\N^r_q}
\newcommand{\pme}{{\pmb{\e}}}
\newcommand{\ped}{\pme_{\mbox{\tiny$\D$}}}

\newcommand{\bm}{{\mathbf{m}}}
\newcommand{\bma}{{\bm^{\ba}}}

\newcommand{\N}{\mathcal{N}}
\newcommand{\U}{\mathcal{U}}

\newcommand{\D}{\mathcal{D}}
\newcommand{\R}{{\sf{Div}}}
\newcommand{\ex}{{\sf{ex}}}
\newcommand{\triv}{{\sf{triv}}}
\newcommand{\Base}{{\sf{Base}}}

\newcommand{\sfk}{\mathsf k}

\begin{document}
\author[S.~M.~Cooper]{Susan M. Cooper}
\address{Department of Mathematics\\
University of Manitoba\\
520 Machray Hall\\
186 Dysart Road\\
Winnipeg, MB\\
Canada R3T 2N2; \it{Email address:} \tt{susan.cooper@umanitoba.ca}}

\author[S.~El Khoury]{Sabine El Khoury}
\address{Applied Mathematics and Computational Sciences Program,
King Abdullah University of Science and Technology,
 Thuwal 23955-6900, Kingdom of Saudi Arabia; \it{Email address:} \tt{khouryss@kaust.edu.sa}}

\author[S.~Faridi]{Sara Faridi}
\address{Department of Mathematics \& Statistics\\
Dalhousie University\\
6316 Coburg Rd.\\
PO BOX 15000\\
Halifax, NS\\
Canada B3H 4R2; \it{Email address:} \tt{faridi@dal.ca}}

\author[S.~Morey]{Susan Morey}
\address{Department of Mathematics\\
Texas State University\\
601 University Dr.\\
San Marcos, TX 78666\\U.S.A.; \it{Email address:} \tt{morey@txstate.edu}}

\author[L.~M.~\c{S}ega]{Liana M.~\c{S}ega}
\address{Division of Computing, Analytics and Mathematics, 
University of Missouri-Kansas City, Kansas City, MO 64110 U.S.A.; \it{Email address:} \tt{segal@umkc.edu}}

\author[S.~Spiroff]{Sandra Spiroff }
\address{Department of Mathematics,
University of Mississippi,
Hume Hall 335, P.O. Box 1848, University, MS 38677
U.S.A.; \it{Email address:} \tt{spiroff@olemiss.edu}}

\setlist{font=\normalfont}



\title{Divisibility Relations and $\D$-Extremal ideals}

\begin{abstract}
A divisibility relation between the generators of a square-free monomial ideal formally encodes the situation when one generator  divides the least common multiple of some other generators. The divisibility relations contribute to the deletion of some parts of the Taylor resolution of the ideal, and therefore lead to finding a resolution closer to the minimal one. Motivated by this observation, for a given set $\D$ of divisibility relations, we study all square-free monomials satisfying the relations in $\D$. We define a class of square-free monomial ideals called $\D$-extremal ideals $\ED$, and show it is optimal in the sense that it is an ideal satisfying exactly those divisibility relations coming from 
$\D$, and no others. We then show that $\ED$ is extremal in the sense  that the resolution and betti numbers of the powers of any square-free monomial ideal satisfying the relations in $\D$ are bounded by those of the same powers of $\ED$.

\end{abstract}

\maketitle

\section{\bf Introduction} 

A free resolution of an ideal is a way to keep track of the relations between the generators of that ideal, and the relations between those relations, and so forth. 
A monomial ideal with $q$ generators has a free resolution -- called the {\it Taylor resolution}~\cite{T}  -- that can be geometrically realized by a simplex with $q$ vertices, each vertex labeled by a monomial generator of the ideal, and each face labeled by the least common multiple (lcm) of its vertex labels. The Taylor resolution, like every free resolution, contains a minimal free resolution of the ideal.  Much of the research in this area is focused on finding methods to delete the extra faces (components) of the Taylor simplex (resolution) in order to arrive at a resolution closer to the minimal one. The main machinery behind such methods is identifying faces with equal lcm labels, which comes down to understanding divisibility relations of the form $m_1\mid \lcm(m_2, \dots, m_s)$ between subsets of monomial generators of an ideal. 

In this paper, we introduce a way to formalize the divisibility relations described above in order to study ideals based on the divisibility relations between their generators, rather than the generators themselves.
 For example, if under a fixed ordering on the minimal monomial generators of an ideal $I$ we know that the monomial generators $m_1,m_2,m_3$  satisfy  $m_1 \mid \lcm(m_2,m_3)$, then we encode this relation via the indices of the monomials as the {\it divisibility relation} $(1,\{2,3\})$  and we say that $m_1,m_2,m_3$ satisfy the relation $(1,\{2,3\})$. Given an integer $q$ and a set $\D \subseteq [q]\times 2^{[q]}$ of divisibility relations encoded as above, we introduce a $q$-generated square-free monomial ideal $\ED$, which we call a $\D$-{\it extremal ideal}, whose minimal generators satisfy the relations in $\D$. 

 One of the reasons for referring to $\ED$ as being extremal (with respect to $\D$) is that its minimal generators, when considered in a specific order, satisfy only the divisibility relations that stem from $\D$.  To state this result, we develop a formal language for working with divisibility relations and, starting with a set $\D$ as above, we define the set $\overline{\D}$ of all divisibility relations that can be {\it deduced} from $\D$ via certain closure operations, see \cref{bar}. In \cref{t:power1-rels} and \cref{t:when-equal} we show:

 \begin{introthm}
The set $\overline{\D}$ consists of all divisibility relations satisfied by the generators of the $\D$-extremal ideal $\ED$, and in particular, $\E_{\overline{\D}}=\ED$.
 \end{introthm}

It is worth noting that we expect that a similar statement  can be formulated for all powers ${\ED}^r$ of $\ED$, namely that ${\ED}^r$ satisfies only divisibility relations deduced from $\D$, however, the main challenge is formulating what it means for divisibility relations to be deduced from a fixed set. The language for such a statement becomes more technical as $r$ increases. The case $r=1$ and its ramifications for free resolutions is the current paper, and the case $r=2$ is handled in  \cite{Blue}.  

$\D$-extremal ideals are a refinement of {\it extremal ideals}, which correspond to the situation when $\D=\emptyset$.
Extremal ideals were introduced in~\cite{Lr} as a class of square-free monomial ideals  whose powers have the largest possible multigraded minimal free resolutions among powers of all square-free monomial ideals $I$. As a result, extremal ideals provide effective bounds on the betti numbers of the powers of $I$. In other words, if $I$ is generated by $q$ square-free monomials, $i\geq 0$, and $r>0$, we have  
$$
\beta_i(I^r) \leq \beta_i(\Erq).
$$ 
The strength of this statement is that the bounds work for {\it all} square-free monomial ideals, and they only depend on  $q$ and $r$, and that the bounds are achieved by the extremal ideal. These bounds can be viewed as significant improvements over the binomial bounds  given by the Taylor resolution, see \cite{Scarf} for a comparison of the bounds when $r=3$.  

$\D$-extremal ideals allow us to give tighter bounds on betti numbers of powers by taking into consideration the divisibility relations among the generators. The following theorem combines statements of \cref{t:EDextremal} and \cref{t:lattice}. 

\begin{introthm}
\label{B} 
Let $I$ be an ideal minimally generated by monomials $m_1,\ldots,m_q$, and $\D$ a set of divisibility relations.
\begin{enumerate}
\item  If $m_1,\ldots,m_q$ are square-free and satisfy all the divisibility relations in $\D$, then  
$$\beta_i(I^r) \leq \beta_i({\ED}^r) \qforall i\geq 0, \quad r>0.$$
\item If $\overline{\D}$ is the set of all divisibility relations satisfied by $m_1,\ldots,m_q$, then $$\beta_i(I)=\beta_i(\ED) \qforall i \geq 0.$$
\end{enumerate}
\end{introthm}

It is worth emphasizing that similar to~\cite{GPW}, the proof of item  (2) above shows that the minimal free resolution of any monomial ideal can be constructed purely from the set of its divisibility relations, or even a smaller set of divisibilities that generates the full set.

To illustrate the improvement  provided by the bounds given by Theorem~\ref{B}, we consider an example. If the ideal $I$ has $q=4$ generators that satisfy $m_1 \mid \lcm(m_2,m_3)$, then, using Macaulay2~\cite{M2} to compute the betti numbers of $\ED$ for $\D=\{(1,\{2,3\})\}$, the bounds in (1) become
$$ \beta_1(I)\leq 5
\qand 
\beta_2(I) \leq 2
\qand 
\beta_i(I)=0 \quad i>2.
$$ 
For comparison, the binomial Taylor bounds (which coincide with the bounds provided by $\E_4$) are: $\beta_1 \leq 6, \beta_2 \leq 4, \beta_3 \leq 1$. Even in this very small example one can see that the existence of a divisibility relation  improves the existing bounds. This improvement typically becomes more significant when examining powers of $I$.  For comparison, the bounds for the betti numbers of $I^2$ starting with $i=1$ are:
$$
{\ED}^2: \   10, 21, 15, 3; 
\quad
{\E_4}^2: \  10,27,32,19,6,1; 
\quad 
\mbox{Taylor}: \  10,45,120,210,252,210,120,45,10,1.
$$
The bounds  provided by ${\ED}^2$
are  clearly better than those offered by  ${\E_4}^2$, and 
significantly better than  Taylor's binomial bounds, which grow extremely quickly for large values of $q$ and $r$.  A further investigation of the strength of the new bounds provided by $\ED$ is pursued in our subsequent paper~\cite{Blue}. 

While finding the betti numbers of ${\ED}^r$ is the topic of future investigations, we describe here the $0$th one, namely the minimal number of generators of ${\ED}^r$: For any $r\ge 1$, we show in \cref{p:minimal} that a minimal generating set of ${\ED}^r$ consists of all products of $r$ (not necessarily distinct)  minimal generators of ${\ED}$. Thus, when $\ED$ is $q$-generated, the powers of $\ED$ achieve the maximum minimal number of generators for the powers of {\it any} $q$-generated ideal, independent from the set $\D$.

The paper is organized as follows. In
\cref{s:closure},  we create a language and a set of operations to study  divisibility relations. In particular, for a set $\D$ of divisibility relations, we develop the notion of a 
{\it closure} $\overline{\D}$, which is the set of all possible  divisibility relations that stem from $\D$. We also introduce a concept of a {\it minimal generating set} of divisibility relations for $\D$, which allows one to work with smaller, more manageable sets.   \cref{sec:D-extremal}
 introduces the $\D$-extremal ideals, and pursues
 the (challenging) task of showing that $\ED$
 satisfies exactly  those divisibility relations in $\overline{\D}$,
 and nothing more. In the last section, we make the connection to betti numbers of powers and we prove Theorem~\ref{B}.

\subsection*{Acknowledgements}
The research for this paper was initiated during the authors' stay at the American Institute of Mathematics (AIM), as part of the AIM SQuaRE program.  We are grateful to AIM for their warm hospitality, and for providing a stimulating  research environment.  Authors Cooper and Faridi are  partially supported by NSERC Discovery Grants 2024-05444 and 2023-05929, respectively.

\section{\bf Divisibility relations}\label{s:closure}

 Given any set of monomials, there are often relations in which one monomial divides the least common multiple of the others. We call such a relation a {\it divisibility relation} among the monomials. When dealing with a monomial ideal, there can be multiple divisibility relations among its  generators.
In this section, we  introduce a formal framework to study these divisibility relations and define operations on them in order to determine additional divisibility relations that can be deduced from a given set of relations.

For example, for
\begin{equation}\label{two relations}  I=(bcg, abg, cdf, adgh, bef)
\end{equation} 
set $u_1 =bcg; u_2=abg; u_3=cdf; u_4 = adgh,  u_5=bef$.  
Then 
\begin{equation}\label{eq:intro-rel}
u_1\mid \lcm(u_2,u_3) 
\qand u_2\mid \lcm(u_4,u_5).
\end{equation}
In general, any monomials satisfying \eqref{eq:intro-rel} will also necessarily satisfy the relations
\begin{equation}
\label{eq:intro-rel-consequences}
u_1\mid \lcm(u_3,u_4,u_5)\qand u_1\mid \lcm(u_2,u_3,u_4).
\end{equation} 
Thus, from a bookkeeping perspective, one needs to record only the divisibilities in \eqref{eq:intro-rel}, which in this paper we will summarize 
 as $(1,\{2,3\})$ and $(2, \{4,5\})$, respectively. 
The additional two divisibilities in \eqref{eq:intro-rel-consequences}, namely $(1,\{3,4,5\})$ and $(1, \{2,3, 4\})$, can be deduced from the original pair, hence in principle they do not need to be recorded. In this section, we will develop a language to distinguish {\it minimal} divisibility relations and {\it non-minimal} ones, that is, those that can be deduced from other ones.

\begin{definition}[{\bf Divisibility relations}]
\label{d: div-rel}
Let $U=\{u_i\st i\in \Lambda\}$ denote a set   of distinct monomials indexed by a finite set $\Lambda$, so that $u_i=u_j$ implies $i=j$. We define a {\bf divisibility relation on} $U$, encoded as the pair $(b,B)\in \Lambda\times 2^\Lambda$, to be a relation of the form 
\begin{equation}\label{e:div-rel}
u_b\mid \lcm(u_i\st i\in B) \qforsome  b\in \Lambda \qand  \emptyset \ne B \subseteq \Lambda.
\end{equation} 
We say that $(b,B)$ is {\bf trivial} if $b\in B$, and we say that $(c,C)$ is an {\bf extension} of $(b,B)$ if $b=c$ and $B \subseteq C$. Let
$(b,B)^\ex$ denote the set of all extensions of $(b,B)$, in other words 
$$(b,B)^\ex =\{ (b,C) \st B \subseteq C\}.$$ 
For a set of divisibility relations $\D = \{(b_1,B_1),\ldots,(b_d,B_d)\}$ the {\bf base set} of $\D$ is the set
$$\Base(\D) = \{b_1,\ldots,b_d\}.$$
The {\bf set of all divisibility relations on $U$} is
\begin{equation}\label{d:Div}
\R(U)=\{(b,B) \st (b,B) \mbox{ is a divisibility relation on } U \} \subseteq \Lambda\times 2^{\Lambda\ssm\{\emptyset\}}.
\end{equation}
The  set of divisibility relations on a monomial ideal $I$, denoted by $\R(I)$, refers to  the set $\R(U)$ where $U$ is the minimal monomial generating set of $I$, considered with a fixed ordering.
\end{definition}
Note that the set $\Lambda$ in this definition will later correspond to the set $[q]$, where $q$ is the number of monomial generators of the ideals considered.
We introduce an operation $\circ$ on $ \Lambda\times 2^\Lambda$. For elements $(b,B)$ and  $(c,C)$ of $ \Lambda\times 2^\Lambda$ define 
\begin{equation}\label{e:circ}
(b,B)\circ (c,C) = (b, (B\smallsetminus \{c\})\cup C).    
\end{equation}

 \begin{example}\label{e:extend}
 With $U=\{u_1,u_2,u_3,u_4, u_5\}$, the divisibility relations \eqref{eq:intro-rel}  generate those in \eqref{eq:intro-rel-consequences} via the operation $\circ$ and extensions, since  
$ (1,\{2,3\}) \circ (2,\{4,5\})=(1,\{3,4,5\})$, and $(1,\{2,3,4\})$ is an extension of $(1,\{2,3\})$.
 
 Note that the operation $\circ$ is not commutative, as we have 
 \[
 (2,\{4,5\}) \circ(1,\{2,3\})=(2,\{2,3,4,5\}).
 \]
 One can also show $\circ$ is not associative, but we leave this to the reader. 
  \end{example}

We now collect some basic properties of divisibility relations and the operation $\circ$ that will be used throughout the paper. 

\begin{proposition}\label{p:new-from-old}
Let $\Lambda$ be a finite set, and $U$ be a set of monomials indexed by $\Lambda$. Then $\circ$ is a (noncommutative, nonassociative) binary operation on $\R(U)$. 
That is, 
 \begin{equation}\label{i3}
 (b,B), (c,C) \in \R(U) \Longrightarrow
 (b,B)\circ (c,C) \in \R(U).
 \end{equation}
In addition, for elements $b,c,d \in \Lambda$
and nonempty subsets $B,C,D$ of $\Lambda$, we have
\begin{enumerate} 
 \item\label{i1} if $(b,B) \in \R(U)$ and $(b,C)\in (b,B)^\ex$, then $(b,C) \in \R(U)$;
 \item\label{i2} $(b,B) \circ (b,B)=(b,B)$;
 \item\label{i4} if $(b,C)\in (b,B)^\ex$, then 
 \begin{align*}
 (b,C)\circ (d,D)\in ((b,B)\circ (d,D))^\ex \qand (d,D)\circ (b,C)\in ((d,D)\circ (b,B))^\ex;
 \end{align*}
 \item\label{i5} $(b,B)\circ (b,C)\in (b,C)^\ex$; 
 \item\label{i6} if $(b,B)$ is trivial, then the following hold: 
 \begin{enumerate}
  \item\label{i6a} $(b,B) \in \R(U)$;
  \item\label{i6b} $(b,B)\circ (c,C)$ is  trivial if $b\ne c$;  
 \item\label{i6c} $(c,C)\circ (b,B) \in (c,C)^\ex$; 
 \item\label{i6d} if $(b,C)\in (b,B)^\ex$, then $(b,C)$ is trivial. 
 \end{enumerate}
 \end{enumerate}
 \end{proposition}
 
 \begin{proof}
 The statements can be verified directly from the definitions.
 \end{proof}

As in \cref{e:extend}, extensions and the operation $\circ$ provide a way to produce new divisibility relations from old ones in a more general setting, as shown in \cref{p:new-from-old} by \eqref{i3} and \eqref{i1}. This is an indication that in order to describe the set $\R(U)$ it would suffice to know a possibly smaller (ideally, minimal) set of relations that can be used to generate all others. We now proceed to define terminology that helps clarify these issues.

Below we define the set $\D^\circ$ of {\it all} divisibility relations formed  by applying $\circ$ to a fixed set $\D$ of divisibility relations, along with the set $\overline\D$, which will be a closure of $\D^\circ$ under extensions and inclusion of trivial relations.
Since $\circ$ is not associative, each time we use the operation on more than two relations we must specify an order. In other words 
$\big ( (b_1, B_1)\circ (b_2,B_2) \big ) \circ (b_3,B_3)$ 
and $(b_1, B_1)\circ \big ( (b_2,B_2)\circ (b_3,B_3) \big )$ are not necessarily equal.
Therefore, the notation $(b_1, B_1)\circ (b_2,B_2) \circ (b_3,B_3)$ without a determined bracket-order is not well-defined on its own. Thus, every time we use this notation we will either refer to a specific bracket order or include all possible bracket orders.

\begin{definition} 
\label{bar} Fix a finite, nonempty set $\Lambda$. For each subset $\D$ of $\Lambda\times 2^\Lambda$ we define  a set $\D^\circ$ that consists of bracketed compositions
\[
(b_1, B_1)\circ \dots \circ (b_s,B_s)
\] 
where $s \ge 1$ and $(b_i,B_i)\in \D$ for all $i\in [s]$, where we include all possible ways of putting brackets in this expression. 
Further, let  
$$
\D^{\ex}= \bigcup_{(b,B) \in \D} (b,B)^\ex
$$ 
be the set of all elements of $\Lambda\times 2^\Lambda$ that are extensions of an element of $\D$ and 
$$\Lambda^{\triv}=\{(b,B)\st (b,B)\in \Lambda\times 2^\Lambda\,\,\text{and}\,\, b\in B\}$$ 
to be the set of trivial relations. Then define
\begin{align}\label{e:closures}
\overline \D& =(\D^\circ)^\ex\cup \Lambda^{\triv}.
\end{align}  
    If $(b,B)\in \overline \D$, we say that $(b,B)$ {\bf can be deduced from} $\D$.
\end{definition} 

 The next result shows that $\D^\circ$ and $\overline \D$ can be understood as closures.

\begin{proposition}
\label{p:closure}
Let $\Lambda$ be a nonempty set. If $\D\subseteq \Lambda\times 2^\Lambda$, the following hold:
\begin{enumerate}
\item $\D \subseteq \D^{\circ} \subseteq (\D^\circ)^\ex\subseteq \overline{\D}$ and $\Base(\D)=\Base(\D^\circ)=\Base((\D^\circ)^\ex)$; 
\item $(\D^{\circ})^\circ=\D^\circ$, $(\D^{\ex})^{\ex}=\D^\ex$ and $(\Lambda^{\triv})^{\ex}=\Lambda^\triv$; 
\item $(\overline \D)^\ex= \overline \D$;
\item if $(b,B), (c,C) \in \overline \D$, then $(b,B)\circ (c,C)\in \overline \D$. In particular, $(\overline \D)^\circ=\overline \D$; 
\item $\overline{\overline{\D}}=\overline\D$;
\item if $\D\subseteq \R(U)$ for a set of monomials $U$ indexed by $\Lambda$, then $\overline{\D}\subseteq \R(U)$. In particular, $\overline{\R(U)}=\R(U)$.
\end{enumerate}
\end{proposition}
\begin{proof}
Statements (1) and (2) follow directly from the definitions of the sets involved.

(3) Using the last two equalities in item (2), where the first is applied to $\D^\circ$, we have
\[
({\overline \D})^\ex=((D^\circ)^\ex)^\ex\cup (\Lambda^\triv)^\ex=(\D^\circ)^\ex\cup \Lambda^\triv=\overline \D.
\]

 (4) Assume $(b,B), (c,C) \in \overline\D$. 
 If $(c,C)\in \Lambda^\triv$, then $(b,B)\circ (c,C) \in (b,B)^\ex$ by \cref{p:new-from-old}~\eqref{i6c}. Since $(b,B) \in \overline{\D}$, then $(b,B)\circ (c,C)\in \overline{\D}$ by item~(3).

 Assume $(b,B)\in \Lambda^\triv$. If $b\ne c$, then by \cref{p:new-from-old}~\eqref{i6b} we have that $(b,B)\circ (c,C) \in \Lambda^\triv\subseteq \overline\D$. If $b=c$, then by \cref{p:new-from-old}~\eqref{i5} we have  $(b,B)\circ (c,C)\in (c,C)^\ex$. Since $(c,C) \in \overline{\D}$, then $(b,B)\circ (c,C)\in \overline{\D}$ by item~(3).
 
 It remains thus to consider the case when $(b,B), (c,C)\in (\D^{\circ})^\ex$. There exist then $(b,B'), (c,C')\in \D^{\circ}$ such that $(b,B)\in (b,B')^\ex$ and $(c,C)\in (c,C')^\ex$. Applying 
 \cref{p:new-from-old}~\eqref{i4} twice, we conclude that $(b,B)\circ (c,C)\in ((b,B')\circ (c,C'))^\ex$, and thus $(b,B)\circ (c,C)\in (\D^\circ)^\ex\subseteq \overline\D$.

 (5) We have 
 \[
 \overline{\overline \D}=(({\overline \D})^\circ)^\ex\cup \Lambda^\triv=(\overline \D)^\ex\cup \Lambda^\triv=\overline \D\cup \Lambda^\triv=\overline \D
 \]
 where the second equality comes from item (4) and the third equality comes from item (3). 
 
 (6) The statement follows from \cref{p:new-from-old}~\eqref{i1},~\eqref{i6a} and~\eqref{i3}. 
\end{proof}

In general, $\overline \D$ is a rather large set. The set $\D^\circ$ can sometimes be described in concrete situations. In particular, $\D^\circ =\D$ under either of the following assumptions: 
\begin{itemize}
    \item $\D$ consists of a single element
    \item  $\D=\{(b_1,B), \ldots, (b_d,B)\}$ where $b_1,\ldots,b_d \notin B$.
\end{itemize}
However, as illustrated in \cref{example2.1}, $\D^\circ$ can still be quite large, but not all of the elements of $\D^\circ$ provide new information. 
We thus introduce below concepts of minimality with respect to divisibility relations. 
By definition, extensions of divisibility relations provide a partial order on the set of divisibility relations that is induced by set inclusion. 
A partial order $\leqslant$ on $\Lambda\times 2^\Lambda$ is given by 
\[
(b,B)\leqslant (c,C) \iff (c,C)\in (b,B)^\ex.
\]
We use this order to define minimality.

\begin{definition}[{\bf Minimality and generating sets}] 
\label{d:minimal}
Fix a finite, nonempty set $\Lambda$ and let $\D \subseteq \Lambda\times 2^\Lambda$.
We say that an element $(b,B)\in \D$ is {\bf minimal}
if it is non-trivial and is minimal with respect to $\leqslant$, that is, if $(c,C)\leqslant (b,B)$ and $(c,C)\in \D$, then $(c,C)=(b,B)$. We denote by $\min(\D)$ the set of elements of $\D$ that are minimal. 
If $U$ is a set of monomials indexed by $\Lambda$, we say that a divisibility relation $(b,B)\in \R(U)$ is {\bf minimal} if $(b,B)\in \min(\R(U))$. 

We say that $\D$ is a {\bf generating set} of divisibility relations on $U$ if $\R(U)=\overline \D$. Moreover, $\D$ is a {\bf minimal generating set} of divisibility relations on $U$ if it is a generating set and 
$$
(b,B)\notin \overline{\D\smallsetminus \{(b,B)\}} \qforeach (b,B)\in \D.
$$
In other words, $\D$ is a minimal generating set for $\R(U)$ if $\D$ is minimal with respect to inclusion among the sets with $\R(U)=\overline\D$.
\end{definition}

\begin{remark}
    Observe that for $\D\subseteq \Lambda\times 2^\Lambda$ we have $\min(\overline \D)=\min(\D^\circ)$. 
    Also, note that for any $U$, there exists a minimal generating set of divisibility relations on $U$, and this set can be chosen so that it consists of minimal divisibility relations.
\end{remark}

We provide below a series of examples to show that the concepts defined in \cref{d:minimal} must be handled with care, since there are subtleties that might not be immediately obvious. In particular, minimal generating sets of divisibility relations are not unique, even when they consist of minimal relations. In these examples, given $\D$, it is always possible to choose monomials $u_1, \dots, u_q$ so that $\R(U)=\overline \D$, which we will show in \cref{s:ED satisfies D}.
The first example shows that it is possible to have $\D^\circ\ne \min(\D^\circ)$, even if $\D=\min(\D)$.

\begin{example} \label{example2.1} Let $U=\{u_1, \ldots,  u_6\}$ and assume $u_1\mid \lcm(u_2, u_3)$, $u_2\mid \lcm(u_4,u_5)$ and $u_5\mid \lcm(u_4,u_6)$, corresponding to the set  
\[
\D=\{(1, \{2,3\}), (2,\{4,5\}), (5,\{4,6\})\}\subseteq \R(U)\,.
\]
Note that $\D=\min(\D)$. 
There are additional divisibility relations on $U$ induced by those in $\D$, including $u_1\mid \lcm(u_3,u_4,u_5)$, corresponding to $(1,\{3,4,5\})\in \D^\circ$, where we note that 
\[
(1,\{3,4,5\})=(1,\{2,3\})\circ(2,\{4,5\}).
\]
Many other relations can be deduced from $\D$ via the $\circ$ operation, including $u_2\mid \lcm(u_4,u_6)$ corresponding to $(2, \{4,5\}) \circ (5, \{4,6\})$ and $u_1\mid \lcm(u_3,u_4,u_6)$ corresponding to $(1, \{2,3\}) \circ (2, \{4,6\})$.  
Thus we have 
\[
\D\cup \{(1,\{3,4,5\}), (1,\{3,4,6\}), (2,\{4,6\})\} \subseteq \D^\circ.
\]
However, $\D^\circ$ also contains elements such as $(1,\{2,3\}) \circ (1, \{3,4,5\})= (1, \{2,3,4,5\}) \in (1,\{3,4,5\})^\ex$,  illustrating that not all elements of $\D^\circ$ are minimal, even if $\D=\min(\D)$.  

Continuing in the manner, $\D^\circ$ is fairly large, however one can check that 
$$\min(\D^\circ)=\D\cup \{(1,\{3,4,5\}), (1,\{3,4,6\}), (2,\{4,6\})\}.$$  
Furthermore, if  $u_i$ are chosen such that $\R(U)=\overline \D$ (e.g. $u_i=\ed{i}$ as in \cref{d:D-def}), then $\D$ is a minimal generating set of divisibility relations on $U$ by \cref{p:closure} (1).
\end{example}

\begin{example}
Assume $U=\{u_1, \ldots, u_6\}$ satisfies $u_1\mid \lcm(u_2,u_3,u_4)$, $u_2\mid \lcm(u_3,u_4)$, describing the set
\[
\D=\{(1,\{2,3,4\}), (2,\{3,4\})\}.
\]
Then $\D\cup\{(1, \{3,4\})\}\subseteq \D^\circ$. We have $(1,\{2,3,4\})\in \min(\D)$ but is not in $\min(\D^\circ)$.  Thus it is possible to have  $\min(\D)\not\subseteq \min(\D^\circ)$. Further, assume that the monomials $u_i$ are chosen so that $\R(U)=\overline \D$. Then both $\D$ and 
\[
\D'=\{(1,\{3,4\}), (2,\{3,4\})\}
\]
are minimal generating sets of divisibility relations on $U$. Thus a minimal generating set of divisibility relations on $U$ is not unique. This example also shows that the relations in a minimal generating set of divisibility relations on $U$ do not need to be minimal in $\R(U)$.
\end{example}

\begin{example}
Assume $U=\{u_1, \ldots, u_6\}$ where $u_i$ have been chosen such that the set
 \[
\D=\{(1,\{2,3\}), (2,\{3,4\}), (4,\{2,3\})\}
\] 
satisfies $\R(U)=\overline \D$. 
Then $(1,\{3,4\})\in \D^\circ$, and observe that both $\D$ and
\[
\D'=\{(1,\{3,4\}), (2,\{3,4\}), (4,\{2,3\})\}
\]
are minimal generating sets of divisibility relations on $U$ consisting of minimal relations. Observe also that $\min(\R(U))=\D\cup \{(1,\{3,4\})\}$ is a non-minimal generating set.
\end{example}

\section{{\bf The $\D$-Extremal ideal $\ED$}}\label{sec:D-extremal}  \label{s:ED satisfies D} 

Introduced in \cite{Lr}, the {\it $q$-extremal ideal} $\E_q$ is a square-free monomial ideal with $q$ generators, which encodes the algebraic behavior of all square-free monomial ideals with $q$ generators, as well as their powers. In this section, we will refine such ideals by adding divisibility relations to the definition. In short, given a set $\D$ of divisibilities between $q$ monomials in any polynomial ring, the $\D$-extremal ideal $\ED$ is a square-free monomial ideal derived from $\E_q$ which, as shown in \cref{t:power1-rels}, satisfies  the divisibility relations in $\D$ and those deduced from $\D$, but no others. We also state in \cref{t:when-equal} equivalent conditions for when two extremal ideals $\E_{\D_1}$ and $\E_{\D_2}$ are equal. 
Before stating the definitions precisely, we first recall the extremal ideals that we wish to fine-tune.

Let $\sfk$ be a field. For an integer $q>0$ we define the polynomial ring $S_{[q]}$ as 
\begin{equation}\label{e:Sq}
S_{[q]}=\sfk[y_A\st \emptyset\ne A \subseteq [q]].
\end{equation}
The {\bf $\pmb{q}$-extremal ideal} is the square-free monomial ideal $\E_q = (\epsilon_1,\ldots, \epsilon_q)$, with $\epsilon_i$ in $S_{[q]}$ defined as 
\begin{equation}\label{e:oldextremal}
  \epsilon_i= \prod_{\substack{\emptyset\ne A \subseteq [q]\\ i \in A}} y_A.
  \end{equation}
  For example, when $q=4$, the ideal $\E_4$ is generated by the monomials $\epsilon_i=y_i y_{ij}  y_{ik} y_{il} y_{ijk} y_{ijl} y_{ikl} y_{ijkl}$, where $i,j,k,l$ are distinct elements in $\{1, 2, 3, 4\}$ and 
  $y_{\{i\}}$ and $y_{\{i,j\}}$ are written $y_i$ and $y_{ij}$, respectively, to simplify notation.

The $q$-extremal ideal is designed to capture the algebraic behavior of  {\it any} square-free monomial ideal with $q$ generators.
We now refine the definition by removing some variables from the generators, so that the resulting ideal satisfies desired divisibility relations.

\begin{definition}[{\bf $\D$-extremal ideals}]\label{d:D-def} 
Let $q>0$ and $d\geq0$.  Set 
\begin{equation}\label{e:oneOK}
\D=\{(b_1,B_1),\ldots,(b_d,B_d)\} \subseteq [q]\times (2^{[q]}\ssm\{\emptyset\}). 
\end{equation}
When $d > 0$, we also write $\D=\big \{(b_j,B_j) \mid j \in [d] \big \}$, and we set $\D = \emptyset$ when $d=0$.
Define 
\begin{align*}
Q(\D) &=\{A\subseteq [q] \, \st  A\neq \emptyset, \mbox{ and } A\cap B\ne\emptyset 
\mbox{ for all } (b,B)\in \D \mbox{ with } b\in A\}\\
&=\{A\subseteq [q] \, \st A\neq \emptyset, \text{ and for all } j\in [d],\   b_j\notin A \text{ or } A\cap B_j\ne\emptyset\}.
\end{align*}
For $i \in [q]$, we  define square-free monomials 
\[
\ed{i}={\displaystyle \prod_{\substack{A\in Q(\D)\\ i\in A}}y_A}\,.
\]
The {\bf $\D$-extremal ideal} is defined as the square-free monomial ideal in the polynomial ring $S_{[q]}$ generated by the $\ed{i}$, namely
$$\ED=(\ed{1}, \ \ldots , \ \ed{q}).$$ 
\end{definition}
Observe that if $\D$ consists only of trivial relations, then $\ED = \E_{\emptyset} = \E_q$. 

\begin{example} \label{mainex} 
When $q=4$, set $\D = \{(1,\{2,3\})\}$. Then $d=1$, $B = \{2,3\}$, and $b_1 = 1$. Since the set $Q(\D)$ consists of all non-empty subsets of $[4]$ except $\{1\}$ and $\{1, 4\}$, and thus the variables $y_1$ and $y_{14}$ do not appear in any generator of the ideal, the ideal $\ED$ in $S_{[4]}$ is generated by the monomials
$$
\begin{array}{llll}
\ed{1}&=y_{12}y_{13}y_{123}y_{124}y_{134}y_{1234}&
\ed{2}&=y_2y_{12}y_{23}y_{24}y_{123}y_{124}y_{234}y_{1234}\\
\ed{3}&=y_3y_{13}y_{23}y_{34}y_{123}y_{134}y_{234}y_{1234}&
\ed{4}&=y_4y_{24}y_{34}y_{124}y_{134}y_{234}y_{1234}.
\end{array}
$$
Note that $\ed{1} \mid \lcm (\ed{2}, \ed{3})$.  
\end{example}

 We now show that when each $B_i$ has at least two elements in \cref{d:D-def}, $\ED$ is minimally generated by $\ed{1},\ldots,\ed{q}$. Thus in this setting, $\R(\ED)$ is defined using this minimal generating set under this ordering. 
 Also, we show that adding redundant relations to $\D$ does not affect the variable set $Q(\D)$, and thus does not change the ideal $\ED$. 

\begin{proposition}\label{p:mingens-simple}
    If $\D$ is as in \eqref{e:oneOK} with the additional assumption that $|B_j| \geq  2$ for each $j \in[d]$, then the following hold:
\begin{enumerate}
\item $\ED$ is minimally generated by $\ed{1},\ldots,\ed{q}$;
\item $Q(\D)=Q(\D\cup\{(b,B)\})$ for  $(b,B)\in \R(\ED)$; 
\item if $(b,B)\in \R(\ED)$ and $b\notin B$, then $b\in \Base(\D)$. 
\end{enumerate}\end{proposition}

\begin{proof} 
 (1) Suppose  $\ed{i}\mid \ed{j}$ for some $i,j\in [q]$ with $i \neq j$.
  If $i \notin \Base(\D)$, define $A=\{i\} \in Q(\D)$. If $i \in \Base(\D)$, then since $|B_u| \ge 2$, for each $u \in [d]$ we fix an element $k_u\in B_u\ssm\{j\}$. Then define $A=(\Base(\D)\ssm\{j\})\cup\{k_1,\ldots,k_d\} \in Q(\D)$. In either case, note that $A \in Q(\D)$ is such that $i\in A$ and $j\notin A$, 
leading to  the contradiction $y_A \mid \ed{i}$ but $y_A\nmid\ed{j}$. 

(2)  The inclusion $Q(\D) \supseteq Q(\D\cup\{(b,B)\})$ is immediate from the definition of $Q(\D)$. 
For the reverse inclusion, assume $A\in Q(\D)$. Consider two cases: if  $b\notin A$, since
$\Base(\D\cup \{(b,B)\})=\Base(\D)\cup \{b\}$, we have $A\in Q(\D\cup \{(b,B)\})$ by definition. Now suppose $b \in A$. Since $(b,B) \in \R(\ED)$ we have $\ed{b}\mid \lcm(\ed{i}\st i\in B)$. Then $y_A\mid \ed{b}$, and hence $y_A\mid \ed{i}$ for some $i\in B$. It follows that $i\in A$, so $A\cap B \ne \emptyset$ . Thus $A \in Q(\D \cup \{(b,B)\})$ and $Q(\D) \subseteq Q(\D \cup \{(b,B)\})$.

(3) Assume $b\notin B$ and $b\notin \Base(\D)$.  Set $A=\{b\}$. Since $b_{j} \not\in A$ for all $j\in [d]$, we have $A\in Q(\D)$. This implies  $A \in Q(\D \cup \{(b,B)\})$ by (1), so $A\cap B\ne\emptyset$ and hence $b\in B$, a contradiction.  
\end{proof}

Our main theorem in this section is that  $\ED$   satisfies exactly those relations deduced from $\D$, and nothing more, as we state formally in \cref{t:power1-rels}. One direction, that $\ED$ satisfies all relations deduced from $\D$, is straightforward to see. The proof of the other direction will require some sophisticated machinery.

\begin{theorem}[{\bf The divisibility relations of  $\ED$}]\label{t:power1-rels} 
 Let $q$, $d$  be positive integers and 
 \begin{equation}
\label{eq:D-def}
 \D=\{(b_1,B_1),\ldots,(b_d,B_d)\} \subseteq [q]\times 2^{[q]} 
 \qwith 
 |B_i| \geq 2 
 \qforall i \in [d].
 \end{equation}
Then $\R(\ED)=\overline \D$. In other words, the set of divisibility relations of $\ED$ consists of all relations that can be deduced from $\D$. 
\end{theorem}

The rest of this section is dedicated to the proof of \cref{t:power1-rels} and some of its consequences. Since we will need to keep track of the indices in $[q]$ and those in $[d]$ that index the elements of $\D$ separately throughout this section, we create an indexing set $\{\lambda_1, \ldots, \lambda_d\}$ of $d$ distinct elements that we will use instead of $1,\ldots,d$ when indexing $\D$. We assume the following setup throughout for positive integers $q$ and $d$. 
\begin{equation}\label{eq:setup}
\begin{aligned}
&\D = \{(b_{\lambda_1},B_{\lambda_1}),\ldots,(b_{\lambda_d},B_{\lambda_d})\}\subseteq [q]\times 2^{[q]} \qwith  |B_i| \geq 2 
 \qforall i \in [d],\\
&\Base(\D) = \{b_{\lambda_1},\ldots,b_{\lambda_d}\},\\
&\nu(b) = \{\lambda_i \st i \in [d], \   b_{\lambda_i}=b\} \qfor b \in \Base(\D).
\end{aligned}
\end{equation}

 To show that $ \R(\ED)=\overline{\D}$, we will construct  a {\it decision} tree $T_{(b,B)}$, which will encode the divisibility relations starting from a given relation $(b,B)$.  As preparation for \cref{c:decision} we remind the reader of some terminology regarding trees. 
A tree is a connected acyclic graph.  A labeled tree is a tree for which each vertex is assigned a label, which is an element of a specific set. The distance between two vertices of a tree is the length (i.e. the number of edges) of the unique path between them.  We say that a tree is {\bf rooted} if one of the vertices is identified as being the root. We say that a vertex of a rooted tree is {\bf odd}, respectively {\bf even}, if its distance to the root (also called the {\bf height} of the vertex) is odd, respectively even. The height of a rooted tree is the largest height of a vertex. If $v$ is a vertex of height $i$ in a rooted tree $T$, then a {\bf child} of $v$ is a vertex $v'$ of height $i+1$ that is adjacent to $v$ in $T$; in this case, we also say $v$ is a {\bf parent} of $v'$.  An {\bf ancestor} of $v$ is a vertex other than $v$ that lies on the path from $v$ to the root.

\begin{construction}[{\bf $\D$-decision tree}] \label{c:decision} Assume the setup in \eqref{eq:setup}.
For a fixed nonempty subset $B$ of $[q]$, and $b \in \Base(\D)$ we construct a rooted tree $T_{(b,B)}$ whose root is labeled $b$, the vertices of even height are labeled by elements  of $\Base(\D)$, and vertices of odd height are labeled by members of the set $\{\lambda_i \st i \in [d]\}$.
 
For a vertex $v$ of $T_{(b,B)}$ of odd height,  let
  $C(v)$ denote the set of labels of all the even ancestors of $v$.
  
Starting from an empty set of vertices we build $T_{(b,B)}$ as described below.

\begin{itemize}
\item (The root: height $0$ vertex) 
Add a vertex $v$ as root and label it $b$. 

\item Starting with $k = 1$ 
\begin{itemize}
\item (adding height $2k-1$ vertices) To each vertex of height $2k-2$ labeled $b' \in \Base(\D)$ add a child labeled $\lambda_i$ for each $\lambda_i\in \nu(b')$.  
\item (adding height $2k$ vertices) 
For each vertex $v$ of height $2k-1$ with label $\lambda_i$, 
\begin{itemize}
    \item if $\emptyset \neq (B_{\lambda_i} \ssm B)  \subseteq \Base(\D)\ssm C(v)$, add a child labeled $b'$ for each $b'\in B_{\lambda_i}\ssm B$; 
    \item otherwise,  the vertex $v$ remains childless and will become a leaf in the output tree.
\end{itemize}
\end{itemize}
\end{itemize}

We call $T_{(b,B)}$ a $\D$-{\bf decision tree} relative to $(b,B)$. 
\end{construction}

\begin{example} 
Let $\D=\{(1,\{2,3\}), (3,\{1,5\}), (3,\{4,5\})\}$ and $B=\{2,5\}$. Thus $\Base(\D)=\{1,3\}$, with $b_{\lambda_1}=1, b_{\lambda_2}=b_{\lambda_3}=3$.
So
$$\nu(1)=\{\lambda_1\}, \quad 
\nu(3)=\{\lambda_2,\lambda_3\}, \quad
B_{\lambda_1}\ssm B = \{3\}, \quad 
B_{\lambda_2} \ssm B =\{1\}, \quad 
B_{\lambda_3} \ssm B =\{4\}.$$
Then we have the following decision trees.

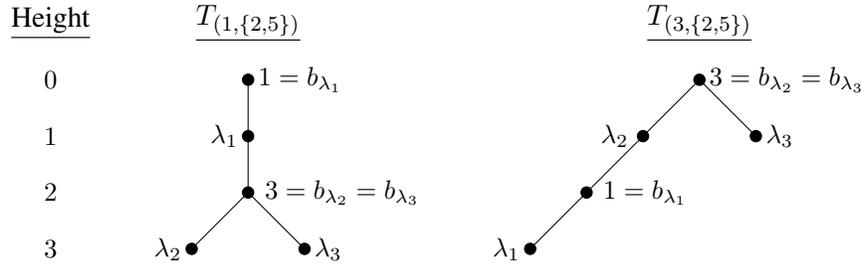
\begin{center}
\begin{tikzpicture}[scale=.75]
\coordinate[label=center:{\underline{Height}}] () at (-3.5,1);
\coordinate[label=center:{\small $0$}] () at (-3.5,0);
\coordinate[label=center:{\small $1$}] () at (-3.5,-1);
\coordinate[label=center:{\small $2$}] () at (-3.5,-2);
\coordinate[label=center:{\small $3$}] () at (-3.5,-3);

\coordinate[label=center:{\underline{$T_{(1,\{2,5\})}$}}] () at (0,1);
\coordinate[label=right:{\small $1=b_{\lambda_1}$}] (A) at (0,0);
\coordinate[label=left:{\small $\lambda_1$}] (B) at (0,-1);
\coordinate[label=right:{\small $\ 3=b_{\lambda_2}=b_{\lambda_3}$}] (C) at (0,-2);
\coordinate[label=left:{\small $\lambda_2$}] (D) at (-1,-3);
\coordinate[label=right:{\small $\lambda_3$}] (E) at (1,-3);

\coordinate[label=center:{\underline{$T_{(3,\{2,5\})}$}}] () at (8,1);
\coordinate[label=right:{{\small $3=b_{\lambda_2}=b_{\lambda_3}$}}] (AA) at (8,0);
\coordinate[label=left:{\small $\lambda_2\,$}] (DD) at (7,-1);
\coordinate[label=right:{\small $\lambda_3$}] (EE) at (9,-1);
\coordinate[label=left:{\small $\lambda_1$}] (BB) at (5,-3);
\coordinate[label=right:{\small $\ 1=b_{\lambda_1}$}] (CC) at (6,-2);

\draw[-] (A) -- (B);
\draw[-] (B) -- (C);
\draw[-] (C) -- (D);
\draw[-] (C) -- (E);
\draw[-] (AA) -- (DD);
\draw[-] (AA) -- (EE);
\draw[-] (CC) -- (DD);
\draw[-] (CC) -- (BB);
\foreach \x in {A,C,AA,CC} 
\draw[black, fill=black] (\x) circle(0.1);
\foreach \x in {B,D,E,BB,DD,EE} 
\draw[black, fill=black] (\x) circle(0.1);
 \end{tikzpicture}
\end{center}
The decision tree stops at height~$3$, since for every vertex $v$ of height~$3$,  $\Base(\D) \ssm C(v)=\emptyset$. 

\end{example}

We now argue that the tree $T_{(b,B)}$ constructed above is finite.  Since all the sets $B_{\lambda_i}$ as well as
$\Base(\D)$ are finite, to show $T_{(b,B)}$ is finite it suffices to show that its height is finite. 
Indeed, for any path from a vertex to the root, the even vertices are labeled by distinct elements of $\Base(\D)$, since each time a new even vertex is added when following \cref{c:decision} its label avoids the labels of its even ancestors. Since the even vertices are labeled by elements of $\Base(\D)$ and $\Base(\D)$ has $d$ elements, a path from the root can have at most $d$ even vertices, and at most $2d$ vertices in total. Thus the tree has height at most $2d$. The construction also dictates that an even vertex has at least one odd child, and thus all leaves of the tree must be odd, and hence the height of the tree is also odd.

For a $\D$-decision tree $T_{(b,B)}$ relative to a relation $(b,B)$, we next define notions of a good vertex and of a bad vertex. We define these recursively, starting with the leaves. Notice that all leaves must be odd, and thus they are labeled by elements of $[d]$.

\begin{definition}\label{d:goodbad}
 Consider a $\D$-decision tree $T_{(b,B)}$. Starting from the leaf vertices, we will assign a value of {\bf bad} or {\bf good} to each vertex $v$ of $T_{(b,B)}$. Below we determine when a vertex is good, and otherwise it is considered bad.
 
\begin{itemize}

\item If $v$ is a leaf labeled  $\lambda_i$, then 
 $v$ is  good if $B_{\lambda_i}\subseteq B$;

 Now assume that a notion of good has been defined for all vertices of height $2k+1 \leq$ height($T_{(b,B)}$).  Then we define: 

\item If $\height(v)=2k$, then $v$ is good if at least one of its children is good; 

\item If $\height(v)=2k-1$ and $v$ is a non-leaf, then $v$ is good if  all its children are good. 
\end{itemize}

We say that a bad vertex $v$ is {\bf extremely bad} if all of its ancestors are bad.

\end{definition} 

\begin{lemma}\label{l:pull-apart}  If $b'$ is the label of a good even vertex  of $T_{(b,B)}$ in  \cref{c:decision},  then  $(b',B)\in \overline \D$. 
\end{lemma}
\begin{proof}  Let  $b' \in \Base(\D)$ be the label of the good vertex $v$ of even height. As $v$ is an even good vertex, it has at least one good child $w$ labeled with some $\lambda_i \in \nu(b')$. In particular, $b_{\lambda_i}=b'$ and $(b',B_{\lambda_i}) \in \D$.

Assuming  $\height(T_{(b,B)})=2k+1$ and $\height(v)=2(k-\ell)$, we prove the statement by induction on $\ell \in \{0,\ldots,k\}$.

The base case is when  $\ell=0$, so $w$ must be a good leaf. Then  $B_{\lambda_i}\subseteq B$, and thus  \[
(b', B)\in (b', B_{\lambda_i})^\ex \subseteq \D^\ex\subseteq (\D^\circ)^\ex\subseteq \overline \D.
\]
Assume now $\ell\ge 1$. If $w$ is a leaf, then by the same argument as above we get  $(b', B)\in \overline \D$. Assume $w$ is not a leaf,   and thus 
$\emptyset \ne B_{\lambda_i}\ssm B\subseteq \Base(\D)$. 
Then all the children of $w$, which are good even vertices of height $2(k-(\ell-1))$, are  labeled with the elements of  
\[
B_{\lambda_i}\ssm B=\{c_1, \dots, c_s\}\, \qforsome s \geq 1.
\]
By the induction hypothesis we have $(c_j, B)\in \overline \D$ for all $j\in [s]$.
Since $(b',B_{\lambda_i})\in \D \subseteq \overline{\D}$, we can construct the divisibility relation
\begin{equation}
\label{element}
(b',B_{\lambda_i})\circ (c_1,B)\circ \dots \circ (c_s,B) \in \overline{\overline{\D}}
\end{equation}
where the order in which the operations are performed is assumed to be from left to right. For example, 
$$
(b',B_{\lambda_i})\circ (c_1, B)\circ (c_2,B)= (b',B_{\lambda_i}\ssm \{c_1\} \cup B) \circ (c_2,B) =
\big( b', \big((B_{\lambda_i}\ssm \{c_1\} \cup B) \ssm \{c_2\}\big) \cup B\big).
$$
Since $c_j \notin B$ for all $j \in [s]$,  the relation in \eqref{element} reduces to   
$$
\big (b', (B_{\lambda_i}\ssm \{c_1,\ldots,c_s\}) \cup B \big )=
(b',B) \in \overline{\overline \D}. 
$$
Since $\overline{\overline \D}=\overline \D$ by \cref{p:closure}, it follows that $(b',B) \in \overline{\D}$.  
\end{proof}

We are now ready to prove \cref{t:power1-rels}, that is, we will show that 
$\R(\ED)=\overline{\D}$ where 
$$
\D = \{(b_{\lambda_1},B_{\lambda_1}),\ldots,(b_{\lambda_d},B_{\lambda_d})\}.
$$ 

\begin{proof}[{\bf Proof of \cref{t:power1-rels}}] We prove each inclusion separately.

\medskip

\noindent{\underline{$\overline{\D} \subseteq \R(\ED)$}.} 
We first show that  $\D \subseteq \R(\ED)$ by proving that 
$$
\ed{{b_{\lambda_j}}}\mid \lcm(\ed{i}\st i\in B_{\lambda_j}) \qforall j\in [d]. 
$$  
    Fix $j\in [d]$. Let $A\in Q(\D)$ be such that $y_A\mid \ed{{b_{\lambda_j}}}$. Then $b_{\lambda_j}\in A$. By the definition of $Q(\D)$, we must have $A\cap B_{\lambda_j}\ne\emptyset$.  Let $i\in A\cap B_{\lambda_j}$. Then the definition of $\ed{i}$ implies $y_A\mid \ed{i}$, and hence $y_A\mid \lcm(\ed{i}\st i\in B_{\lambda_j})$. This implies $\D\subseteq \R(\D)$.     
    The rest now follows from \cref{p:closure}~(6). 

\medskip

\noindent{\underline{$\R(\ED) \subseteq \overline{\D}$}. }
Let $(b,B)\in \R(\ED)$. If $(b,B)$ is trivial, then $(b,B) \in [q]^\triv \subseteq \overline{\D}$. Thus we assume $(b,B)$ is not trivial. Then $b \notin B$, so using \cref{p:mingens-simple}~(3), we see that  $b\in \Base(\D)$. If the root of the decision tree $T_{(b,B)}$ is good, then from \cref{l:pull-apart} we have $(b,B) \in \overline{\D}$ and we are done. It suffices thus to assume that the root of $T_{(b,B)}$ is bad. We show below that this is not possible.

Since $(b,B)$ is a divisibility relation on $\ED$, \cref{p:mingens-simple}~(2) implies that for every $\emptyset \ne A \subseteq [q]$ we have: 
\begin{equation}
\label{e:A}
  A\in Q(\D) \qand b\in A \quad \implies \quad A\cap B\ne \emptyset.
\end{equation}
We will obtain a contradiction by constructing a set $A\in Q(\D)$ such that $b\in A$  but $A\cap B=\emptyset$. 

 For each bad leaf $v$ (which is necessarily odd) with label $\lambda_i$ we have $\emptyset\ne B_{\lambda_i}\ssm B\not\subseteq \Base(\D)\ssm C(v)$, in view of \cref{d:goodbad} and \cref{c:decision}. So we define an element $\mu(v)\in [q]$ such that 
\[
\mu(v)\in B_{\lambda_i}\ssm B \qand \big(\mu(v)\notin \Base(\D)\qor \mu(v)\in C(v)\big)\,.
\]
Now define $A=A_1\cup A_2 \subseteq [q]$ where 
\begin{align*}
A_1&=\{b' \st b' \text{ is the label of an extremely bad even vertex in } T_{(b,B)}\}\\
A_2&=\{\mu(v)\st v \text{ is an extremely bad leaf in } T_{(b,B)}\}. 
\end{align*}

We claim that the set $A$ contradicts  \eqref{e:A}, which will finish our argument. 
Indeed, by construction, we have $A\cap B=\emptyset$. Also, note that $b\in A_1 \subseteq A$, since $b$ is the label of the root, which is bad and hence extremely bad. 

It remains to argue that $A\in Q(\D)$. Assume $b'\in A \cap \Base(\D)$. We need to show that for every $\lambda_i \in \nu(b')$,  $B_{\lambda_i}\cap A\ne\emptyset$.

If $b'\in A_2\ssm A_1$,  then  $b'=\mu(v)$ for an extremely bad leaf $v$ of $T_{(b,B)}$. By definition of $\mu(v)$, and since we know $b' \in \Base(\D)$, we must have  $b'=\mu(v)\in C(v)$. Since $v$ is extremely bad, all its ancestors are extremely bad, so 
 then $b'=\mu(v)\in A_1$, a contradiction. 

Therefore  $b'\in A_1$, which means that  
$b'$ is the label of an extremely bad even vertex $v$. Since $v$ is bad, all its children are also bad. Thus, for an element $\lambda_i \in \nu(b')$, the vertex $v$ has a bad child $v'$ with label $\lambda_i$. Note that $v'$ is extremely bad, since it is bad and its parent $v$ is extremely bad.  
If $v'$ is a leaf, then $\mu(v')\in A\cap B_{\lambda_i}$. If $v'$ is not a leaf, then it has at least one bad child $v''$ with label $b''$, where $b''\in B_{\lambda_i}\ssm B$. The ancestors of $v''$ are $v'$, $v$ and all the ancestors of $v$, and hence $v''$ is also extremely bad. Then $b''\in A\cap B_{\lambda_i}$. We have shown that $A\cap B_{\lambda_i} \neq \emptyset$, which ends our argument.
\end{proof}

As a result of the previous theorem, we now have a criterion showing when a set of divisibility relations 
describes  {\it all} possible  relations on {\it some} ideal.

\begin{corollary}
\label{c:whenD}
For $q\ge 1$, a set $\D$ as in \eqref{eq:D-def} is equal to $\R(I)$ for some ideal $I$ minimally generated by $q$ monomials if and only if $\overline{\D}=\D$. 
\end{corollary}

\begin{proof}
One implication follows directly from \cref{t:power1-rels} using $I = \ED$. The other implication follows from \cref{p:closure}~(6).
\end{proof}

We conclude this section with a discussion  of when two sets of divisibility relations define the same extremal ideals. Based on \cref{d:D-def}, the key to determining the $\D$-extremal ideals are the sets $Q(\D)$. 
Recall from \cref{d:minimal} that the minimal relations of $\D$ are those that are not trivial and which do not properly extend another relation of $\D$. We now show that the $\D$-extremal ideal is unchanged by the closure operations studied so far.

\begin{lemma}
\label{E-equal}
If  $\D$ as in \eqref{eq:D-def} is a set of divisibility relations, then 
\[
\ED=\E_{{\D^\circ}}=\E_{\overline \D}=\E_{\min(\D)}=\E_{\min(\overline{\D})} = \E_{\min(\D^\circ)}\,.
\]
\end{lemma}
\begin{proof} 
First note that if $\D_1$ and $\D_2$ are two sets of relations, then
\begin{equation}\label{e:D1D2}
\D_1\subseteq \D_2 \Longrightarrow Q(\D_2)\subseteq Q(\D_1), \qand
\E_{\D_1}=\E_{\D_2} \iff Q(\D_1)=Q(\D_2).
\end{equation}
By \cref{p:mingens-simple}~(2) we have $Q(\D)=Q(\D\cup \R(\ED))$, hence   \cref{t:power1-rels} gives $Q(\D)=Q(\overline{\D})$. 
Then \eqref{e:D1D2} implies $Q(\D)=Q(\D^\circ)$ and $\ED=\E_{{\D^\circ}}=\E_{\overline \D}$. 

We now show $\ED=\E_{\min(\D)}$. It suffices to show $Q(\D)=Q(\min(\D))$. In view of \eqref{e:D1D2}, we need to show the inclusion $Q(\min(\D))\subseteq Q(\D)$. Let $A\in Q(\min(\D))$. We need to show that, if $(b,B)\in \D$ is nontrivial and $b\in A$ then $A\cap B\ne\emptyset$. We know there exists $(b,C)\in \min(\D)$ such that $C\subseteq B$. Then, since $A\in Q(\min(\D))$ and $b\in A$, we have $A\cap C\ne\emptyset$, and hence $A\cap B\ne\emptyset$. This shows thus $\ED=\E_{\min(\D)}$.
In particular, we also have $\E_{\overline\D}=\E_{\min(\overline\D)}$ and $\E_{\D^\circ}=\E_{\min(\D^\circ)}$.
\end{proof}

We conclude this section with our second main result, which provides equivalent ways of determining when two sets of divisibility relations define the same extremal ideal.

\begin{theorem}
\label{t:when-equal}
For a positive integer $q$ and $\D_1$ and $\D_2$ as in \eqref{eq:D-def}, the following statements are equivalent: 
\begin{enumerate}
\item $\E_{\D_1}=\E_{\D_2}$;
\item $\R(\E_{\D_1})=\R(\E_{\D_2})$;  
\item$\overline{\D_1}=\overline{\D_2}$;
\item $\min(\overline{\D_1})=\min(\overline{\D_2})$.
\end{enumerate}
\end{theorem}
\begin{proof}
 The implications (1)$\implies$(2) and (3)$\implies$(4) are clear. The implication (2)$\implies$(3) follows from \cref{t:power1-rels}. Finally, the implication (4)$\implies$(1) follows from  \cref{E-equal}. 
\end{proof}

\section{$\E_\D$ is extremal}

We now show why the ideals $\ED$ are extremal relative to the divisibility relations in $\D$ via a family of maps. For each square-free monomial ideal $I$ whose generators satisfy the relations in $\D$, we show that the map $\psi_I$ defined in \cite[Definition 7.5]{Lr} sends the generating set of $\ED$ onto a generating set of $I$, which will allow for comparisons of the algebraic behaviors of powers of $\ED$ and powers of $I$. In particular, we show that the minimal number of generators of  the powers of the ideals $\ED$ is maximum possible  (\cref{p:minimal}) and, more generally, the betti numbers of their powers bound the betti numbers of the powers of any square-free monomial ideal satisfying the relations in $\D$ (\cref{t:EDextremal}). Moreover, the minimal free resolution of any monomial ideal can be modeled using an ideal $\ED$ for an appropriately chosen set $\D$ (\cref{t:lattice}).

We now determine minimal generating sets for the powers $\EDr$ of $\ED$.  This result extends \cref{p:mingens-simple}, which deals with the case $r=1$. In what follows, we define 
$$
\Nrq = \{(a_1,\ldots,a_q) \in \NN^q \st a_1+\cdots + a_q=r\}
$$
and for $\ba=(a_1, \dots, a_q)\in \Nrq$, we set 
$$\ped^\ba=\ed{1}^{a_1}\cdots \ed{q}^{a_q}.$$

\begin{proposition}[{\bf A minimal generating set for $\EDr$}]
\label{p:minimal}
    Given $q,r>0$, let $\D$ be as in \eqref{eq:D-def}. 
   For $\ba,\bc\in \Nrq$, if $\ped^{\ba}\mid \ped^{\bc}$, then $\ba=\bc$. In particular, the monomials  $\ped^{\ba}$ with $\ba\in \Nrq$ form a minimal generating set for  $\EDr$. 
\end{proposition}

\begin{proof}
Let $\D = \{(b_u, B_u) \mid u \in [d]\}$ and assume $\ped^{\ba}\mid \ped^{\bc}$. For each $A\in Q(\D)$, let  $a_A$ and $c_A$ denote the power of the  variable $y_A$ in $\ped^{\ba}$ and $\ped^{\bc}$, respectively. So for 
$A \in Q(\D)$ we have $a_A \leq c_A$ where
$$
a_A = \sum_{i \in A}a_i \qand 
c_A=\sum_{i \in A} c_i .
$$
Since  for $A \in Q(\D)$ we have $a_A \leq c_A$, and also $\displaystyle \sum_{i\in [q]}a_i=r=\sum_{i\in [q]}c_i$, it immediately follows that 
\begin{equation}\label{eq:innotin}
\sum_{i \in [q]\ssm A}a_i \geq  
\sum_{i \in [q]\ssm A} c_i
\qforall
A \in Q(\D).
\end{equation}
If $A=\{i\}$ for some $i\in [q]\ssm \Base(\D)$, observe that $A\in Q(\D)$. Since $a_A\leq c_A$ we have thus
\begin{equation}
\label{eq:inotinB}
a_i\le c_i 
\qforall i\in [q]\ssm \Base(\D).
\end{equation}

For each  $j \in [q]$ we define 
$$A_j= \Base(\D)\cup \bigcup_{u\in [d]} (B_u\ssm\{j\}) \qand A'_j=A_j\ssm\{j\}.$$ 

Let $j\in [q]$. For each $u \in[d]$, the hypothesis $|B_u|\ge 2$ in \eqref{eq:D-def} implies  $B_u\cap A_j\ne\emptyset$ and $B_u\cap A'_j\ne\emptyset$, and hence $A_j, A'_j\in Q(\D)$. Since $\Base(\D)\subseteq A_j$, and hence $$[q]\ssm A_j\subseteq [q]\ssm \Base(\D),$$ we combine \eqref{eq:innotin} with $A=A_j$ and \eqref{eq:inotinB} to  obtain 
\begin{equation}
\label{eq:combined}
a_i=c_i\qforall i\in [q]\ssm A_j.
\end{equation}

If $j\notin \Base(\D)$, then $j\notin A_j$. Applying \eqref{eq:combined} with $i=j$ yields $a_j=c_j$. Thus 
\begin{equation}
\label{eq:jnotinB}
a_j=c_j \qforall j\in [q]\ssm \Base(\D).\\
\end{equation}

If $j\in \Base(\D)$, we apply \eqref{eq:innotin} with $A=A'_j$, noting that $[q]\ssm A'_j=([q]\ssm A_j)\cup \{j\}$, to get 
\begin{equation*}
\label{eq:A'}
a_j+\sum_{i\in [q]\ssm A_j}a_i\ge c_j+\sum_{i\in [q]\ssm A_j}c_i\,.
\end{equation*}
In this equation, we cancel the equal terms given by \eqref{eq:combined} to obtain $a_j\ge c_j$. We have shown thus
\begin{equation}
\label{eq:jinB}
a_j \geq c_j  \qforall j\in \Base(\D).
\end{equation}

Now recalling that  $\displaystyle\sum_{i\in [q]}a_i=r=\sum_{i\in [q]}c_i$, and using \eqref{eq:jnotinB}  we can write  
$$\sum_{i \in \Base(\D)} a_i = \sum_{i \in  \Base(\D)} c_i,$$
which along with \eqref{eq:jinB} forces
\begin{equation}\label{eq:jinBeq}
a_j = c_j 
\qforall j\in \Base(\D).
\end{equation}
From \eqref{eq:jnotinB} and \eqref{eq:jinBeq} it follows that $\ba=\bc$, and we are done.
\end{proof}

Now that we have a minimal generating set for $\EDr$ for all $r$, we introduce a family of ring homomorphisms that we will use in showing that $\EDr$ is extremal. Using the setting as in \cref{d:D-def}, fix $\D=\{(b_j,B_j) : j \in [d]\}$ and 
let $I$ be any ideal generated by square-free monomials $m_1, \dots, m_q$ in $R = \sfk[x_1, \ldots, x_n]$ satisfying the divisibility relations given by $\D$. 
Working with the variables in the ring $R$ and the generators of $I$, for each $k\in [n]$, set 
\[
\mathcal{A}_k= \{j \in [q] \st x_k \mid m_j \}.
\]
Since $m_{b_j}\mid \lcm(m_i \st i\in B_j)$, we have $x_k\mid m_{b_j}\implies x_k\mid m_b$ for some $b\in B_j$ and thus 
\begin{equation}
\label{Ak}
b_j\in \mathcal{A}_k
\implies 
B_j\cap \mathcal{A}_k\ne \emptyset
\implies
\mathcal{A}_k\in Q(\D) \qforall k\in [n]. 
\end{equation}

With the polynomial ring $S_{[q]}=\sfk[y_A\st \emptyset\ne A \subseteq [q]]$ as before, recall from \cite[Definition 7.5]{Lr} that the ring homomorphism $\psi_I: S_{[q]} \rightarrow R$ is defined
by 
 \[
\psi_I(y_A)=\begin{cases}
  \displaystyle\prod_{\substack{k\in[n] \\ A=\mathcal{A}_k}} x_k, & \mbox{if } A=\mathcal{A}_k \mbox{ for some }
   k \in [n],\\ 1, & \mbox{otherwise}.
  \end{cases}
\]

We next collect some properties of $\psi_I$ that will be used to show that the betti numbers of $\EDr$ are maximal among those of powers of square-free monomial ideals satisfying the relations of $\D$. Note that this proposition extends \cite[Lemma 7.7]{Lr}.

\begin{proposition}\label{p:old-new}
Let $I$ be an ideal of the polynomial ring $R$.  If $I$ is generated by square-free monomials $m_1,\ldots,m_q$ satisfying the divisibility relations given by $\D$ as in \eqref{e:oneOK}, then
   \begin{enumerate}
   \item $\psi_I(\ed{i})=\psi_I(\e_i)=m_i$\, for $i \in [q]$;
   \item $\psi_I(\ped^\ba)=\bma$ for each $\ba\in \Nrq$, where $\bma = m_1^{a_1} \cdots m_q^{a_q}$;
   \item $\psi_I(\EDr)R  =I^r$ for every $r>0$;
   \item $\psi_I$ preserves least
     common multiples, that is: 
$$
   \psi_I(\lcm(\ped^{\ba_1}, \dots, \ped^{\ba_t}))=\lcm(\bm^{\ba_1}, \dots, \bm^{\ba_t}) \qforall \ba_1,\ldots,\ba_t \in \Nrq, \quad t\ge 1.
$$     
   \end{enumerate}

\end{proposition}

\begin{proof}
    (1) We know $\psi_I(\e_i) =m_i$ from \cite[Lemma 7.7]{Lr}. To show the remaining equality,  note that \eqref{Ak} implies that for $A\notin Q(\D)$, we have $A\ne \mathcal A_k$ for all $k$, and thus $\psi_I(y_A)=1$. Thus we have
    \[
  \psi_I(\ed{i})=\psi_I\big({\prod_{\substack{A\in Q(\D)\\ i\in A}}y_A}\big)=\prod_{\substack{A\in Q(\D)\\ i\in A}}\psi_I(y_A)=\prod_{\substack{A\subseteq [q]\\ i\in A}}\psi_I(y_A)=\psi_I\big(\prod_{\substack{A\subseteq [q]\\ i\in A}}y_A\big)=\psi_I(\e_i)\,.
  \]

Statement (2) follows directly from (1), since $\psi_I$ is a ring homomorphism, and (3) is a direct consequence of (2).

(4) Set ${\ba_i}=(a_{i1}, a_{i2}, \dots, a_{iq})$, for $i\in [t]$. Note that 
\[
\lcm(\ped^{\ba_1}, \dots, \ped^{\ba_t})=\prod_{A\in Q(\D)} (y_A)^{\underset{1\le i\le t}{\max}  \sum_{j\in A} a_{ij}}\quad\text{and}\quad \lcm(\pme^{\ba_1}, \dots, \pme^{\ba_t})=\prod_{\emptyset\ne A\subseteq [q]} (y_A)^{\underset{1\le i\le t}{\max}  \sum_{j\in A}  a_{ij}}.
\]
Now we proceed as in (1), using the observation that $\psi_I(A)=1$ when $A\notin Q(\D)$ and the fact that $\psi_I$ is a ring homomorphism, to conclude 
\[
\psi_I(\lcm(\ped^{\ba_1}, \dots, \ped^{\ba_t}))=\psi_I(\lcm(\pme^{\ba_1}, \dots, \pme^{\ba_t}))=\lcm(\bm^{\ba_1}, \dots, \bm^{\ba_t}),
\]
where the last equality was proved in \cite[Lemma 7.7(iii)]{Lr}. 
\end{proof}

\begin{corollary}
If $I$ is an ideal generated by $q\ge 1$ square-free monomials satisfying the divisibility relations given by $\D$ as in \eqref{eq:D-def}, then 
\[
\R(\EDr)\subseteq \R(I^r)\qforall r\ge 1. 
\]
\end{corollary}

 We are now ready to explain why the ideals $\ED$ are referred to as being extremal. This language comes from the fact that the betti numbers of powers of $\ED$ provide upper bounds for the betti numbers of any square-free monomial ideal satisfying the relations $\D$. In view of \cref{p:old-new}, the proof of our next result is virtually identical to the one given for \cite[Theorem 7.9]{Lr}, which deals with the case when $\D=\emptyset$. 

 Recall that if $I$ is a monomial ideal with minimal generators $m_1, \dots, m_q$, its {\bf lcm lattice}, denoted $\LCM(I)$, is the lattice with elements labeled by the least common multiples of the generators, ordered by divisibility.  The atoms of this lattice are the monomials $m_1, \dots, m_q$ and the lcm of elements in the lattice is their join, i.e.~their least common upper bound in the lattice.

\begin{theorem}[{\bf $\ED$ is extremal}]
\label{t:EDextremal}
Let $I$ be a square-free monomial ideal whose minimal generators satisfy the divisibility relations given by $\D$ as in \eqref{eq:D-def}. Then for all integers $r \ge 1$ and $i \ge 0$ we have
\begin{enumerate}
  \item if a simplicial complex $\Delta$ supports a free resolution of
   ${\ED}^r$, then $\psi_I(\Delta)$ supports  a free resolution of $I^r$;
   
\item  if $\m \in \LCM(I^r)$, then 
$$\displaystyle \beta_{i, \m}(I^r) \le
  \sum_{\tiny \substack{\pme\in \LCM(\EDr)\\
  \psi_I(\pme)=\m}}
  \beta_{i,\pme}(\EDr);$$
\item $\beta_i(I^r)\leq \beta_i(\EDr)$.
  \end{enumerate}
\end{theorem}
 
\begin{proof}
 By~\cref{p:old-new}~(3), $\psi_I(\EDr)R=I^rR$ and by~\cref{p:old-new}~(4), least common multiples are preserved. Thus \cite[Lemma 3.6]{extremal} holds when $\Erq$ is replaced by $\EDr$. The proof now follows as in \cite[Theorem 3.8]{extremal}. 
\end{proof}

 In the case of the first power of the $\D$-extremal ideal, we also state stronger results that can be applied to all monomial ideals.
 The final theorem shows that, given an integer $q\ge 1$, knowledge of the minimal free resolutions of the extremal ideals $\ED$ for all possible minimal (in the sense of \cref{d:minimal}) sets $\D\subseteq [q]\times 2^{[q]}$ as in \eqref{eq:D-def} results in knowledge of the minimal free resolutions of {\it all} ideals minimally generated by $q$ (not necessarily square-free) monomials. In particular, any possible behavior of a free resolution of a monomial ideal can be modeled using an appropriately chosen $\D$-extremal ideal. This idea  has been put to use in~\cite{Reid}  towards constructing examples of minimal free resolutions of monomial ideals that are dependent on characteristic.  

We will use the following result of Gasharov et.~al~\cite[Theorem 3.3]{GPW}.
 If $I$ and $I'$ are monomial ideals such that there exists a map $f\colon \LCM(I)\to \LCM(I')$ that is bijective on the atoms and preserves joins, and $\mathbf F$ is a minimal graded free resolution of $I$, then \cite[Construction 3.2]{GPW} describes a graded complex $f(\mathbf F)$ which is obtained from $\mathbf F$ via a relabeling.  It is proved in \cite[Theorem 3.3]{GPW} that $f(\mathbf F)$ is a free resolution of $I'$ and, if $f$ is an isomorphism of lattices, then $f(\mathbf F)$ is minimal as well.

\begin{theorem} 
\label{t:lattice}
 Let  $I$ be an ideal minimally generated by $q \geq 1$ monomials, $\D$ a set of divisibility relations such that $\D\subseteq \R(I)$. 
 Then there exists a map $$f\colon \LCM(\ED)\to \LCM(I)$$
 such that $f$ is bijective on the atoms, preserves lcms (joins), and maps the minimal multigraded free resolution $\mathbf F_\D$ of $\ED$ to a multigraded free resolution $f(\mathbf F_\D)$ of $I$. Furthermore, if $\overline \D=\R(I)$  then $f$ is an isomorphism of lattices and thus  $f(\mathbf F_\D)$ is the minimal multigraded free resolution of $I$, and in particular $\beta_i(\ED)=\beta_i(I)$ for all $i\ge 0$. 
\end{theorem}

\begin{proof} Assume  $I$ is minimally generated by $m_1, \dots, m_q$.
If $\D\subseteq \R(I)$, it follows that $\overline{\D}\subseteq \overline{\R(I)}$ and hence
    \begin{equation}
    \label{e:RD}
    \R(\ED)=\overline{\D}\subseteq \overline{\R(I)}=\R(I),
    \end{equation}
    where the first equality comes from \cref{t:power1-rels} and the second equality comes from \cref{p:closure}~(6). 
      We define $f\colon \LCM(\ED)\to \LCM(I)$ such that $f$ sends the lattice element labeled by $\lcm(\ed{i_1}, \dots, \ed{i_s})$ to the lattice element labeled by $\lcm(m_{i_1}, \dots, m_{i_s})$, for all $1\le i_1<\dots<i_s\le q$ and $s\in [q]$. The fact that $f$ is well-defined is a straightforward consequence of the inclusion $\R(\ED)\subseteq \R(I)$ of \eqref{e:RD}.  Furthermore, if $\overline \D=\R(I)$, then equality holds in \eqref{e:RD} and thus $\R(\ED)=\R(I)$, implying that the map $f$ is an isomorphism of lattices. The conclusions about $f(\mathbf F_\D)$ follow from \cite[Theorem 3.3]{GPW}. 
\end{proof}


\end{document}